\documentclass[a4paper,12pt,reqno,oneside]{amsart}
\usepackage[utf8]{inputenc}
\usepackage[finnish,english]{babel}
\usepackage{amsthm,amssymb,mathtools,hyperref,amsaddr,amsmath}
\title[scattering relation to dn map in low regularity]{Determination of DN map from scattering relation for simple surfaces in low regularity}

\author{Kelvin Lam}
\email{klam0008@uw.edu}
\address{Department of Mathematics, University of Washington}
\date{\today}

\newtheorem{theorem}{Theorem}

\newtheorem{lemma}[theorem]{Lemma}
\newtheorem{corollary}[theorem]{Corollary}

\theoremstyle{definition}
\newtheorem{definition}[theorem]{Definition}
\newtheorem{remark}[theorem]{Remark}

\DeclareMathOperator{\Op}{Op}
\DeclareMathOperator{\Id}{Id}

\newcommand{\Vol}{\mathrm{Vol}}

\newcommand*{\N}{\mathbb{N}}
\newcommand*{\Z}{\mathbb{Z}}

\newcommand*{\R}{\mathbb{R}}

\newcommand{\der}{\mathrm{d}}
\newcommand{\eps}{\varepsilon}

\newcommand{\abs}[1]{\left\vert#1\right\vert}
\newcommand{\norm}[1]{\left\lVert#1\right\rVert}
\newcommand{\doo}[1]{\partial_{\mathrm{#1}}}

\newcommand{\antti}[1]{}
\newcommand{\kelvin}[1]{}
\newcommand{\joonas}[1]{}

\begin{document}

\maketitle

\begin{abstract}
In this paper we prove that on a simple surface where the metric is $C^{17}$, the scattering relation determines the Dirichlet to Neumann map (DN map) - a result proved in \cite{PesUhl2005} for the case when the metric is smooth. For metrics with finite differentiability we had to modified each technical result used in the original proof; such as properties of the exit time function and the characterization of $C_\alpha$ space (Theorem 5.1.1 \cite{PSUbook}). Moreover, surjectivity of $I^*$ in the original proof required the use of microlocal analysis of the normal operator $I^*I$; which is not a standard pseudodifferential operator when the metric only has finite regularity- this was addressed in \cite{IKL2023}. Finally, using the injectivity of $I$ on Lipschitz one forms for simple $C^{1,1}$ manifolds by \cite{IlmKyk2023} we prove an equivalent characterization of harmonic conjugacy using operators determined by the scattering relation (Theorem 1.6 \cite{PesUhl2005}) to prove the titular result. We also prove that the boundary distance function determines the metric at the boundary (which in turns determines the scattering relation) for a closed disk even when the metric is only $C^{1,1}$ and the exponential map is only Lipschitz and does not preserve tangent vectors or differentials pointwise.
\end{abstract}

\section{Introduction}

Given two simple metrics $g$ on a compact surface with boundary $M$, it is proved in \cite{PesUhl2005} that the boundary distance function determines the metric up to a boundary fixing diffeomorphism; in other word simple surfaces are boundary rigid. The key step in the proof of boundary rigidity involves showing that the boundary distance function determines the Dirichlet-to-Neumann (DN map) for simple metrics. In this paper we prove that the same is true when the metric has sufficiently high but finite regularity.

We prove several intermediate results with various regularity requirements for the metric $g$, from which it follows that for simple $C^{17}$ metrics DN map is determined by the boundary distance function.

The three main results in this paper are the following: 1. For $C^{1,1}$ metrics on the closed disk $\mathbb{D}$, the boundary distance function determines the metric on the boundary up to a boundary fixing gauge (which fixes the boundary distance function). 2. The $I^*$ the adjoint of the geodesic x ray transform is surjective for simple metrics $g$ that are $C^{17}$. 3. Given a simple metric $g$ of class $C^{17}$ on a compact surface with boundary, the boundary distance function determines the DN map.
These results are summarized in the section below.

\subsection{Main results}

\begin{theorem}
\label{thm1}
Let $M \subset \mathbb{R}^2$ be the closed unit disk, $g_1, g_2 $ be $C^{1,1}$ metrics on $M$ such that $d_{g_1}|_{\partial M \times \partial M} = d_{g_2}|_{\partial M \times \partial M}$, then there exists a $C^{2,1}$ diffeomorphism $\psi: M \rightarrow M$ with $\psi |_{\partial M} = id$ such that $\psi^*g_1 |_{\partial M} = g_2 |_{\partial M}$ 
\end{theorem}

\begin{theorem}
\label{thm2}
Let $(M,g)$ be a simple surface with $g \in C^k$ with $k \geq 10$, let $f \in C^l(M)$ with $1 < m+1< l-1 < k-7 $, $m,l,k \in \mathbb{N}$,  then there exists $w \in C^{\min(k-4,m)} (\partial_+ SM)$ with $w^\# \in C^{\min(k-4,m)} (SM)$, such that
$I^* w = f$.

\end{theorem}

\begin{theorem}
\label{thm3}
Let $(M,g_1)$ and $(M,g_2)$ be a simple surfaces with $g_1, g_2 \in C^{17}$, with $d_{g_1} = d_{g_2}$, then the DN maps $\Lambda: C^{2,\alpha} (\partial M) \rightarrow C^{1,\alpha}(\partial M)$ determined by $g_1,g_2$ are equal for all $0 < \alpha < 1$.
\end{theorem}

The proof of theorem 1 relies on some recent results by \cite{Minguzzi} and \cite{anderson}, theorem 2 relies on some microlocal analysis at low regularity studied in \cite{marschall1996} and \cite{IKL2023}, all of which derived from the proofs in \cite{PSUbook}, and finally we use theorem 1, 2 and a modification of the proof in \cite{PesUhl2005} to prove the titular result.

\subsection{Acknowledgements}

The author was supported by NSF, and partially supported by HKUST. The author would like to thank his advisor Gunther Uhlmann for his guidance, patience and for suggesting this problem. John M. Lee for his generosity with his time and knowledge. And Gabriel P. Paternain, Hart F. Smith, Joonas Ilmivirta, Antti Kykk\"anen, Joey Zou and Haim Grebnev for discussions.\\

\section{Preliminaries}
\label{sec:preliminaries}

We introduce the geometric preliminaries and the operators used throughout this paper in this section.

\subsection{Simple manifolds}

\begin{definition}
\label{def:simple}
Let $k \in \Z$ and assume that $k \ge 2$. Let $M$ be a compact smooth manifold with a smooth boundary and equip $M$ with a $C^k$ smooth Riemannian metric $g$. We say that $(M,g)$ is simple if $M$ is $C^k$-diffeomorphic to the closed Euclidean unit ball in $\R^n$ and the following hold:
\begin{enumerate}
    \item The boundary $\partial M$ is strictly convex in the sense of the second fundamental form.
    \item The manifold is non-trapping i.e. all geodesics hit the boundary in a finite time.
    \item There are no conjugate points in $M$.
\end{enumerate}
\end{definition}

When the Riemannian metric $g$ is $C^\infty$-smooth definition~\ref{def:simple} is equivalent to any standard definition of a simple manifold.

\subsection{Function spaces}

Let $(M,g)$ be a simple manifold where $g \in C^k(M)$ for some $k \ge 2$. Since $M$ is $C^{k-1}$-diffeomorphic to the closed Euclidean unit ball $\overline{B} \subseteq \R^n$ we take $M = \overline{B}$ from now on and all computations are to be interpreted via a $C^{k-1}$-diffeomorphism as explained in (Theorem 3.8.5, \cite{PSUbook})

We use smooth global coordinates $(x^1,\dots,x^n)$ in the definitions of our functions spaces. We use the Riemannian volume for $\der\Vol_g$ to define $L^2(M)$ in the standard way i.e. $L^2(M) = L^2(M,\der\Vol_g)$.

For $s>0$ we denote by $H^s_c(M)$ the space of compactly supported functions in $H^s(M)$ (Here by compactly supported we mean $f = \phi f$ for some $\phi \in C_c^\infty(M)$). For $s > 0$ we let $H^{-s}(M)$ be the continuous dual of $H^s(M)$ and $H^{-s}_c(M)$ be the subspace of compactly supported distributions.

\subsection{Non smooth operators and symbol}
we recall the basics of a non-smooth pseudodifferential calculus introduced in~\cite{marschall1996}. We rerecord the results that are relevant to the current work for the convenience of the reader.

Let $m \in \R$ and $r,L \in \N$ be given.
Multi-indices in~$\N^n$ are denoted by~$\alpha$ and~$\beta$. For all $\rho,\delta \in [0,1]$ the symbol class~$S^m_{\rho\delta}(r,L)$ consists of continuous functions $p \colon \R^n \times \R^n \to \R$ satisfying the estimates
\begin{equation}
\abs{\partial_\xi^\alpha p(x,\xi)}
\le
C_\alpha
(1 + \abs{\xi})^{m - \rho\abs{\alpha}}
\end{equation}
and
\begin{equation}
\norm{\partial_\xi^\alpha p(\,\cdot\,,\xi)}_{C^r_\ast}
\le
C_{\alpha r}
(1 + \abs{\xi})^{m + r\delta - \rho\abs{\alpha}}
\end{equation}
for all $\abs{\alpha} \le L$.

Given a symbol $p \in S^m_{\rho\delta}(r,L)$ the corresponding operator~$\Op(p)$ is defined by its action
\begin{equation}
\Op(p)f(x)
=
\int_{\R^n}
e^{ix\cdot\xi}
p(x,\xi)
\hat f(\xi)
\,\der \xi
\end{equation}
on functions~$f$ in~$L^2(\R^n)$. The identity operator~$\Id$ is the operator corresponding to the constant symbol~$1$.

\subsection{Geodesic X-ray transforms}

Let $(M,g)$ be a simple manifold where $g \in C^k(M)$ for some $k \ge 2$. For a given unit vector $v \in T_xM$ there is a unique geodesic $\gamma_{x,v}$ corresponding to the initial conditions $\gamma_{x,v}(0) = x$ and $\dot\gamma_{x,v}(0) = v$. Since the manifold is non-trapping, the geodesic $\gamma_{x,v}$ is defined on a maximal interval of existence $[-\tau_-(x,v),\tau_+(x,v)]$ where $\tau_\pm(x,v) \ge 0$ and we abbreviate $\tau \coloneqq \tau_+$.

The X-ray transform $If$ of a function $f \in L^2(M)$ is defined for all inwards pointing unit vectors $(x,v) \in \doo{+}SM$ by the formula
\begin{equation}
If(x,v)
\coloneqq
\int_0^{\tau(x,v)}
f(\gamma_{x,v}(t))
\,\der t.
\end{equation}
For $g \in C^5$, the same proof of Prop 4.1.2 in \cite{PSUbook} works to show that $I: L^2(M) \rightarrow L^2(\partial_+ SM)$ is bounded (The proof relies on the continuity of the term defined in lemma 3.2.8, which requires the odd extension of $\tau$ being $C^1$ in lemma 3.2.6, which requires $g \in C^5$). The backprojection $I^\ast h$ of a function $h$ on $L^2(\doo{+}(SM)))$ is defined for all $x \in M$ by the formula
\begin{equation}
I^\ast h(x)
\coloneqq
\int_{S_xM}
h(\phi_{-\tau(x,-v)})
\,\der S_x(v).
\end{equation}

Finally, we define the operator $N$ which we will call the normal operator. The normal operator is defined on $L^2(M)$ by the formula
\begin{equation}
Nf(x)
=
2\int_{S_xM}
\int_0^{\tau(x,v)}
f(\gamma_{x,v}(t)))
\,\der t
\,\der S_x(v).
\end{equation}
It is proved in [Prop 8.1.5 \cite{PSUbook}] that $N$ agrees with the composition $I^* I$ on $L^2(M)$ and hence the name normal operator.\\
In the case that $M$ is diffeomorphic a closed ball (which is the case if the metric $g$ is $C^k$ simple for $k \geq 2$). Then we can also consider the operator $\phi I^*I \phi$ with $\phi \in C_c^\infty(M)$, it is shown in lemma 11 of \cite{IKL2023} (See \ref{lemma11} below) that $\phi I^*I \phi$ is actually pseudodifferential operators with non smooth symbols.

%The composition of the X-ray transform and the backprojection operator is well-defined on $L^2(M)$. We call this operator the normal operator and denote $N \coloneqq I^\ast I$. Unraveling the definitions we see that
%\begin{equation}
%Nf(x)
%=
%2\int_{S_xM}
%\int_0^{\tau(x,v)}
%f(\gamma_{x,v}(t)))
%\,\der t
%\,\der S_x(v)
%\end{equation}
%for all $f \in L^2(M)$. Again, $N$ is well-defined on $H^s(M)$ for $s > 0$ we extend the action of $N$ to fractional Sobolev spaces $H^{-s}(M)$ by duality.

\section{Boundary distance function determines the scattering relation}

In this section we prove theorem 1 (\ref{thm1}) for $C^{1,1}$, then we show that for $C^2$ metrics the boundary distance function also determines the scattering relation.

    We first prove two technical lemmas before proving theorem 1, for the rest of the proof of theorem 1 we assume $M = \mathbb{D} \subset  \mathbb{R}^2$.

\begin{lemma}\label{lemma1}

    Let $g$ be a $C^{1,1}$ metric on $M$, there exists a $C^{2,1}$ diffeomorphism $\phi : U \rightarrow U'$ with $U,U'$ neighborhoods of $\partial M$ such that $\phi|_{\partial M} = id$ and $\phi^*\nu = \Bar{\nu}$ where $\nu$ is the unit normal vector with respect to $g$ and $\Bar{\nu}$ is the Euclidean normal on $\partial M$.

\end{lemma}

\begin{proof}

    Consider $U \subset \mathbb{R}^2 = (a,b) \times (1-\epsilon, 1+\epsilon)$ a neighborhood of a segment of $\partial M$ in polar coordinates $\psi(\theta,r)$ for small $\epsilon < 1$ and $b-a < 2\pi$.
    Suppose such a $C^{2,1}$ diffeomorphism exists for $U$, then
    
\begin{equation}        
    \nu=d\phi(\Bar{\nu}) = d\phi(dr)= \frac{1}{|dr-\frac{g_{12}}{g_{11}} d\theta|_g} dr - \frac{\frac{g_{12}}{g_{11}}}{|dr-\frac{g_{12}}{g_{11}} d\theta|_g} d\theta
\end{equation}

    where $g_{ij}$ is the metric component of $g$ in the coordinates of $\psi$. In this case the differential $d\phi$ at the boundary (at $r=1$) must be of the form

\begin{equation}\label{eq5}
    d\phi |_{\partial M} = \begin{pmatrix}
        \frac{d\phi_1}{d\theta} & \frac{d\phi_1}{dr}  \\ \frac{d\phi_2}{d\theta}& \frac{d\phi_2}{dr}
    \end{pmatrix}|_{\partial M} = \begin{pmatrix}
         1 & \frac{-g_{12}}{|dr-\frac{g_{12}}{g_{11}} d\theta|_g}   \\ 0 & \frac{1}{|dr-\frac{g_{12}}{g_{11}} d\theta|_g}
    \end{pmatrix} |_{\partial M}
\end{equation}
    \\

    We will construct a diffeomorphism on a collar neighborhood of $\partial M$ in $\mathbb{R}^2$ with differential \ref{eq5}. Observe that the differential above has coefficients that are $C^{1,1}$, by [lemma 3.3.1, \cite{anderson}] we may choose

    $\Tilde{\phi_1}, \Tilde{\phi_2} \in C^{2,1}(M)$ such that $\Tilde{\phi_1}|_{\partial M} = \Tilde{\phi_2}|_{\partial M} = 0$ and 
    
\begin{equation}
\frac{\partial \Tilde{\phi_1}}{\partial r} = \frac{-g_{12}}{|dr-\frac{g_{12}}{g_{11}} d\theta|_g} , \frac{\partial \Tilde{\phi_2}}{\partial r} = \frac{1}{|dr-\frac{g_{12}}{g_{11}} d\theta|_g}
\end{equation}
    
    Now define a $C^{2,1}$ map $\phi$ on $U := (0,2\pi] \times (1-\epsilon,1+\epsilon)$ in polar coordinates for some small $\epsilon$ by $\phi(\theta,r) = (\phi_1,\phi_2) =:  (\theta + \Tilde{\phi_1} , 1+\Tilde{\phi_2})$, then $\phi$ fixes $\partial M$ ($r=1$) and maps $U$ into a neighborhood of $\partial M$ in $\mathbb{R}^2$ with the differential (number) at the boundary. Since $\Tilde{\phi_2}(\theta,1) = 0 $ and $\frac{\partial \Tilde{\phi_2}}{\partial r}(\theta, 1) > 0$  for any fixed $\theta$, $\Tilde{\phi_2}(\theta, r) <0 $ for sufficiently small $1-r>0$, by compactness of $\partial M$ we may choose sufficiently small $\epsilon$ so that $\phi$ maps $[0,2\pi) \times (1-\epsilon , 1]$ into $M$. Furthermore, since the differential at the boundary is clearly invertible, $\phi$ is a local diffeomorphism near $r=1$. Hence for a sufficiently small $\epsilon$, $\phi$ is a local diffeomorphism that maps $U$ into a neighborhood of $\partial M$ in $M$ which fixes $\partial M$ and $\phi^*(\nu) = \Bar{\nu}$.\\

    An argument similar to that of the proof of (Theorem 5.25 \cite{irm}) shows that $\phi$ is injective on an possibly even smaller neighborhood of $\partial M$, so it restricts to a $C^{2,1}$ diffeomorphism from some neighborhood $U$ of $\partial M$ to another such neighborhood $U'$.

\end{proof}

\begin{lemma}\label{lemma2}

Suppose the $C^{1,1}$ metrics $g_1$ and $g_2$ on $M$ induces the same boundary distance functions, then the metrics agree in the tangential direction at the boundary.

\end{lemma}

\begin{proof}
    It suffices to show that the boundary distance function determines the metric in the tangential direction. In other word, we prove that given a $C^{1,1}$ metric $g$ on $M$, $d_g|_{\partial M \times \partial M }$ determines $g|_{\partial M}$ in the tangential direction. Let $p \in \partial M$ and $v \in T \partial M$, and a smooth curve $\tau: (-\epsilon, \epsilon) \rightarrow \partial M$ with $\tau(0) = p$ and $\tau'(0) =v$. 

    Consider a local coordinate $((x,y),U)$ centered at $p$ with $\partial M \subset \{y=0\}$ near $p=0$. By applying the appropriate linear transformations, we may assume $g$ is euclidean at $0$. Denote $\Bar{g}$ the euclidean metric in the local coordinates $((x,y),U)$, then we know $\Bar{g}|_0 = g|_0$.

    Consider  $\lim_{s \rightarrow 0} \frac{d(p, \tau(s))}{s}$, then since $\frac{d(p, \tau(s))}{s} \leq \frac{\int_0^s |\tau'(t)|_g dt}{s}$ and $|v|_{g} = \lim_{s \rightarrow 0}  \frac{\int_0^s |\tau'(t)|_g dt}{s}$. We have 
    
    \begin{equation}
    \lim_{s \rightarrow 0} \frac{d(p, \tau(s))}{s} \leq   |v|_{g}
    \end{equation}

    We now prove the equality. For any $(x,y) \in U$, consider the change of basis matrix from the $\frac{\partial}{\partial x}, \frac{\partial}{\partial y}$ basis to the orthonormal basis with respect to $g$, denoted as $T(x)$, then we have for any $(x,v) \in TM$, 
    
\begin{equation}
    |v|_{\bar{g}} \leq ||T^{-1}(x)|| |v|_{g}
\end{equation}

where $||T^{-1}(x)||$ is the operator norm of $T^{-1}$. Since $T(x) \rightarrow I$ for $x \rightarrow 0$, $||T^{-1}(x)|| \rightarrow 1$. So for a fixed $\epsilon$, there is a sufficiently small neighborhood $U'$ near $0$ so that 

\begin{equation}
    |v|_{\Bar{g(x)}} \leq |v|_{g(x)}(1+\epsilon)
\end{equation}

    for all $(x,v) \in TU'$.

    By (Theorem 6 in \cite{Minguzzi}), for $C^{1,1}$ metrics, for every normal neighborhood $N$ of a point $p$, every absolutely continuous curve starting from $p$ connecting to another point $q \in N$ must have length larger than the geodesic connecting them. So choose a geodesic ball $V$ of small radius center at $p$ (Theorem 4 in \cite{Minguzzi}),  as a consequence of Theorem 6 in \cite{Minguzzi} all length minimizing geodesics connecting $p$ with points in $V$ lie in $V$. Since we are taking limits of $s$ with $\tau(s)$ converging to $p$ we may assume $U$ lie in such a $V$. So we have 

\begin{equation}
    (1+\epsilon) \frac{d(0,\tau(s))}{s} = \frac{\int_0^{\alpha}(1+\epsilon)|\gamma'_s(t)|_g dt}{s} \geq \frac{\int_0^\alpha |\gamma'_s (t)|_{\Bar{g}}}{s}
\end{equation}

Where $\gamma_s : [0,\alpha] \rightarrow V$ is the g-geodesic connecting $p$ with $\tau(s)$. But in Euclidean metrics the shortest curve between two points must be a straight line, since $\tau$ is a straight line lying on $y=0$, we must also have $$\frac{\int_0^\alpha |\gamma'_s (t)|_{\Bar{g}}}{s} \geq \frac{\int_0^s |\tau'(s)|_{\Bar{g}} dt}{s}$$ So we have the following chains of inequalities $$\limsup_{s \rightarrow 0}{(1+\epsilon) \frac{d(0,\tau(s))}{s}}  \geq \limsup_{s \rightarrow 0}{\frac{\int_0^s |\tau'(s)|_{\Bar{g}} dt}{s}} =\lim_{s \rightarrow 0}{\frac{\int_0^s |\tau'(s)|_{\Bar{g}} dt}{s}} = |v|_g$$
    
    But since $\epsilon$ was arbitrary, we have $\limsup_{s \rightarrow 0} \frac{d(0,\tau(s))}{s} \geq |v|_g$.

    So we have $$|v|_g \geq \limsup_{s \rightarrow 0} \frac{d(0,\tau(s))}{s} \geq |v|_g$$. Since we also know $\liminf_{s \rightarrow 0}{\frac{\int_0^s |\tau'(s)|_{\Bar{g}} dt}{s}} = \lim_{s \rightarrow 0}{\frac{\int_0^s |\tau'(s)|_{\Bar{g}} dt}{s}}$ (since $\lim_{s \rightarrow 0}{\frac{\int_0^s |\tau'(s)|_{\Bar{g}} dt}{s}}$ converges), the same inequalities as above holds true if we replace $\limsup$ with $\liminf$, so we have $$\limsup_{s \rightarrow 0} \frac{d(0,\tau(s))}{s} = \liminf_{s \rightarrow 0} \frac{d(0,\tau(s))}{s} = |v|_g$$ which shows $$\lim_{s \rightarrow 0} \frac{d(0,\tau(s))}{s} = |v|_g$$

    This shows that $|v|_g$ is completely determined by the distance function for a $C^{1,1}$ metric.

\end{proof}
\hfill \\

    We are now ready to prove theorem 1:

\begin{proof}[Proof of theorem 1]

    By \ref{lemma1} there exists $\phi : U \rightarrow U'$ that is a boundary fixing diffeomorphism between neighborhoods $U,U'$ of $\partial M$ such that $\phi(\nu) = \Bar{\nu}$. Using the the exponential map with resepct to the Euclidean metric on $M = \mathbb{D}$, the proof of (Prop 11.2.5 \cite{PSUbook}) remains valid for $C^2$ diffeomorphism near the boundary, from which we can conclude there exists a $C^{2,1}$ diffeomorphism $\Phi: M \rightarrow M$ that restricts to $\phi$ near the boundary (The regularity of this diffeomorphism will be one order higher than that of the metrics). Furthermore, $\Phi$ is a diffeomorphism such that $\Phi*(\nu_1) = \Bar{\nu}$. We can find another such $\Phi_2$ so that $\Phi_2^*(\nu_2) = \Bar{\nu}$, then $\Phi=: \Phi_2^{-1} \Phi_1$ is also boundary fixing and $\Phi^*(\nu_1) = \nu_2$. Since boundary distance functions are invariant under boundary fixing diffeomorphism, $\Phi^*g_1$ and $g_2$ has the same unit normal vector field at $\partial M$ and boundary distance function , and by \ref{lemma2} they agree in the tangential direction at the boundary.
\end{proof}

\begin{corollary}
     Suppose $g_1$ and $g_2$ are two simple $C^3$ metrics with the same boundary distance function on a simple manifold, then they have the same scatter relation (defined above). 
\end{corollary}

\begin{proof}
    
    It is proved in \cite{PSUbook} that $C^k$ simple manifolds are $C^{k-1}$ diffeomorphic to closed ball, so we may without loss of generality assume $g_1$ and $g_2$ are two $C^2$ metrics on a closed ball. the abscence of conjugate points , non trapping and strict convexisty are $C^2$ diffeomorphism invariant conditions, so $g_1$ and $g_2$ are $C^2$ simple metrics on the euclidean disk. So we can apply (Lemma 11.3.2 \cite{PSUbook}) to conclude that the scattering relations and exit time functions are equal and (Lemma 11.2.6 \cite{PSUbook}) to conclude that the volume form are equal. 
\end{proof}

\begin{remark}

For $C^{1,1}$ metrics on a closed disk, if we assume the non trapping condition and strict convexity (which are defined for $C^{1,1}$ metrics, and furthermore if we assume there exists $x$ in the interior of $M$ with so that $ exp_x: D_x \rightarrow M$ is a lipscthiz homeomorphism, then (Lemma 11.3.2, \cite{PSUbook}) applies almost everywhere to conclude that the scattering relations are equal almost everywhere. For Simple $C^{1,1}$ simple manifold (Defined in \cite{IlmKyk2021}), the Santal\'{o} formula holds by (lemma 24, \cite{IlmKyk2023}), so $g_1|_{\partial M} = g_2|_{\partial M}$ and $\tau_{g_1} = \tau_{g_2}$ together implies that the Vol($M$, $g_1$) = Vol($M$,$g_2$).

\end{remark} \hfill \\

\section{Surjectivity of the backprojection operator}

We now prove main theorem 2 following a modification of the argument in (Theorem 8.2.5 \cite{PSUbook}). Throughout the rest of the paper we will assume $(M,g)$ is a simple surface.

Similar to (Lemma 3.1.8 in \cite{PSUbook}), we embed $M$ into a closed manifold isometrically of the same dimension with metric also in $C^k$ for $k \geq 2$. Cover $N$ with finitely many simple open sets $M_j$ with $M \subset U_1$ and $M \bigcap \Bar{U_j}$ for $j \geq 2$, and consider a smooth partition of unity $\phi_j$ subordinate to this cover. We now consider the operator $A := L^2(N) \rightarrow L^2(N)$ defined by $Af = \sum_j^n \phi_j I_j^*I_j\phi_j f$, where $I_j$ is the geodesic X-ray transform for the simple manifold $\Bar{M_j}$ for each $j$.
We first we state several technical lemmas from (\cite{IKL2023}):

\begin{lemma}[\cite{marschall1996} Theorem 2.1.]
\label{lemma6}
Let $p \in S^m_{\rho\delta}(r,L)$ and consider the operator $P \coloneqq \Op(p)$. Suppose that $\rho,\delta \in [0,1]$ and $r,L > 0$ satisfy
\begin{equation}
\delta
\le
\rho,
\quad
L
>
\frac{n}{2},
\quad
r
>
\frac{1 - \rho}{1 - \delta}
\frac{n}{2}.
\end{equation}
Then the operator $P \colon H^{s+m}(\R^n) \to H^s(\R^n)$ is bounded when
\begin{equation}
(1-\rho)
\frac{n}{2}
-
(1-\delta)
r
<
s
<
r.
\end{equation}
\end{lemma}

\begin{lemma}[\cite{IKL2023}, Lemma 11]
\label{lemma11}
Let $(M,g)$ be a simple manifold with $g \in C^k(M)$ for some $k \ge 5$. Then for each $j$ the operator $\phi_j I_j^*I_j\phi_j$  belongs to $S^{-1}(k-s,s-4)$ for all $s \in [4,k]$ with $4 \le s \le k$.
\end{lemma}

\begin{lemma}[\cite{IKL2023} Lemma 20]
\label{lemma20}
Let $(M,g)$ be a simple manifold with $g \in C^k(M)$ for some $k \ge 7 + \frac{n}{2}$. Consider the operator $B := \phi I^*I \phi$ where $\phi \in C_c^\infty (M)$. Then there is an operator $P$ (That is, a left parametrix for $B$) and $\eps > 0$ so that $PB = \Id + R$
where $\Id$ is an operator acting as the identity on elements in $H^{t-\tau}(\R^n)$ which are supported in the set where $\psi = 1 = \phi$ and the remainder
\begin{equation}
R
\colon
H^{t-\tau}(\R^n) 
\to 
H^t(\R^n)
\end{equation}
is continuous whenever $0 < \tau \le \eps$ and
\begin{equation}
-k+6+\frac{n}{2}
<
t
<
k-6-\frac{n}{2}.
\end{equation}
\end{lemma}

 We need to first prove some properties of $A$.

\begin{theorem}\label{thm12}

    (a) Let $k>10$ and $ 7-k < 0 < t < l-1 < k-7 $, then if $Au =f$ with $f \in H^l(N)$, and $u \in L^2(N)$, then $u \in H^{l-1}(N)$.\\ 
    (b) $Au=f$ for $f \in L^2(N)$ has a solution $u \in H^{-1}(N)$ iff  $\langle   f, w \rangle_{L^2}  =0$ for all $w \in Ker(A^*)$

\end{theorem}

\begin{proof}

To prove (a), apply \ref{lemma20} to the operator $\phi_j I_j^*I_j\phi_j$ and obtain the parametrix $L_j: H^1 \rightarrow L^2$ \ref{lemma6} such that  
then $L_j \phi_j I_j^*I_j\phi_j u = \phi_j u + R_j\phi_j u = L_j \phi_j f \in H^{l-1}(\Bar{U_j})$ where $R_j: H^{t-\tau}(\Bar{U_j}) : H^{t}(\Bar{U_j})$ for all $-k+7 < t < k-7$ and a small $\tau < 1 \in \mathbb{Q} $, so that we can conclude that $\phi_j u \in H^{0+\tau}(\Bar{U_j})$, by a bootstrapping-like argument we can then conclude that $\phi_j u \in H^{l-1}(\Bar{U_j})$, which implies $u \in H^{l-1}(N)$.
\\

To prove (b), we consider the space $$Y := \{ f\in L^2(N) | \langle f, w \rangle_{L^2} = 0 \ \forall \ w \in ker(A^*)     \}$$

(where $A^*$ is the $L^2$ adjoint).

We will show that the range of $A$ is surjective onto $Y$. Note that $A: H^{-1} \rightarrow L^2$ is bounded By \ref{lemma11} and \ref{lemma6} (Also see remark preceding remark 12 in \cite{IKL2023}). Given a fixed $w \in ker(A^*)$, for any test functions $u \in L^2$ (So $u \in H^{-1}$) we have: $$\langle Au , w \rangle_{L^2}  = \langle u , A^*w \rangle_{L^2} = 0  $$

which means $Au \in Y$.

    Equip $Y$ with the $L^2$ inner product. Suppose the range of $A: H^{-1} \rightarrow L^2$ is not dense in Y, then by orthogonal projection there is an element $f \in Y$ such that $\langle f , Au \rangle_{L^2} = 0$ for all $ u \in H^{-1}$, in particular any $u \in L^2$. But this means for all $u \in L^2$ we have $\langle A^*f , u \rangle_{L^2} = 0$, so $f \in ker(A^*)$, by definition of $Y$ we then have $\langle f, f \rangle_{L^2} = 0$ so $f= 0$.

Now we show $A$ has closed range in $Y$; we will show that there exists some $C>0$ such that for all $u \in H^{-1}(N)$ with $u \perp Ker(A)$, we have $|u|_{H^{-1}} \leq C|Au|_{L^2}$. Suppose not, then by increasing choices of $C$ and normalizing $|u|_{H^{-1}}$ we obtain a sequence of $u_i$ such that $|u_i|_{H^{-1}} = 1$ and $|Au_i|_{L^2} \rightarrow 0$. Apply the operator $L:= \sum_j^n \phi_j L_j \phi_j : L^2(N) \rightarrow H^{-1}(N)$(remark preceding lemma 17 and lemma 6 in \cite{IKL2023} and \ref{lemma6}) to $Au_i \in L^2(N)$ and obtain $u_i + \sum_j^n \phi_j R_j \phi_j u_i \rightarrow 0 \in H^{-1}$, since $|u_i|_{H^{-1}}$ are bounded, and each $R_j$ are compact operators, we get from Rellich theorem \cite{folland} that for each $j$ there exists a sub-sequence $u_{i_k}$ so that each $\phi_j R_j \phi_j u_{i_k}$ converges in $H^{-1+\tau}$ for some small positive $\tau$, since there's only finitely many $j$ this gives a sub-sequence $u_{i_k}$ such that $\sum_j^n \phi_j R_j \phi_j u_{i_k}$ converges in $H^{-1+\tau}$ and hence also in $H^{-1}$. Then we have $u_{i_k}$ also converges to some $u \in H^{-1}$.
For any test functions $\psi$ (suffice to take $\psi \in L^2$), then consider $\langle u , A^* \psi \rangle$ (as distributional pairing). This makes sense since  $A^*$ is the $L^2$ adjoint of $A$ which is just equal to $A$ since $A: L^2 \rightarrow L^2$ is self adjoint, which means $A^*$ is also one Sobolev degree smoothing($A$ is by \ref{lemma11})) so $A^* \psi  \in H^1$, and we have:

$$\langle u , A^* \psi \rangle = \lim \langle u_{i_k} , A^* \psi \rangle =  \lim \langle Au, \psi \rangle = 0 $$

This shows that $u \in Ker(A)$, but since each $u_{i_k} \perp Ker(A)$ we have $u \perp Ker(A)$ by continuity of inner product, so $u \in Ker(A)$ and $u \perp Ker(A)$ so $u = 0$, but this contradicts with $|u|_{H^{-1}}=1$, so we are done.

This shows that there exists some $C>0$ such that for all $u \in H^{-1}(N)$ with $u \perp Ker(A)$, we have $|u|_{H^{-1}} \leq C|Au|_{L^2}$, let $u_i$ be any sequence such that $Au_i$ converges in $Y$, then consider the $\Tilde{u_i} := u_i - proj_{ker(A)} u_i $, then $A(\Tilde{u_i}) = A(u_i)$, so $Au_i$ being Cauchy implies $\Tilde{u_i}$ is Cauchy, let $u$ be the limit of $u_i$, then $Au = \lim{A u_i}$. so indeed $A$ has closed range in $Y$.

So $A$ has closed range that's dense in $Y$, so (b) follows.

\end{proof}

\begin{proof}[Proof of surjectivity of $A$]
    We are now in a position to prove that $A: L^2(N) \rightarrow H^1(N)$ is in fact surjective. We do so by first proving that $A: L^2 \rightarrow L^2$ is injective; suppose $Af = 0$, then $\langle Af , f \rangle_{L^2} =0$, then by definition of $A$ we have 
    $$Af = \sum_j^n \phi_j I_j^*I_j\phi_j f = 0$$

so

$$0 = \langle \sum_j^n \phi_j I_j^*I_j\phi_j f , f \rangle_{L^2(N)} = \sum_j^n \langle I_j^*I_j\phi_j f , \phi_j f \rangle_{L^2(\Bar{U_j})}$$ $$= \sum_j^n \langle I_j\phi_j f , I_j \phi_j f \rangle_{L^2(\Bar{U_j})} = \sum_j^n |I_j \phi_j f|^2_{L^2(\Bar{U_j})}$$
    
    so each $I_j \phi_j f =0$, by the injectivity of $I_j$ on $L^2$ \cite{IKL2023}, we have that $\phi_j f$ is 0 for all $j$ so $f = 0$.\\

    This shows that $A$ is injective on $L^2$, since $A: L^2(N) \rightarrow L^2(N)$ is self adjoint, $A^*$ is also injective, by \ref{thm12} (b) we have $A: L^2 (N) \rightarrow H^1(N)$ is surjective.
\end{proof}

\begin{proof}[Proof of main theorem 2]:

Let $f \in C^l(M)$ with $1 < m+1 < l-1 < k-7$, $m,l,k \in \mathbb{N}$. Extend $f$ to $C^l(N)$ and still denote it $f$, so that it is in $H^l(N)$, so that in particular $f \in H^1(N)$. By the preceding result there exists $h \in L^2(N)$ such that $Ah = f$, by \ref{thm12}(a) since $f \in H^l(N)$, $h \in H^{l-1}(N)$, by Sobolev embedding $h \in C^m(N)$. 

Define 
\begin{equation}
w_1 := I_1 \phi_1 h = \int_0^{\tau_1(x,v)} \phi_1 (h (\varphi_{1,(x,v)}(t)) dt
\end{equation}

where $\tau_1$ and $\varphi_{(1, \cdot) }$ are the exit time function and geodesic flow with respect to $\Bar{M_1}$. The geodesic flow of a $C^k$ metric has $k-1$ regularity so $\varphi_1 \in C^{k-1}(\partial_+SM_1)$. By an argument identical to that in (Theorem 3.2.6 \cite{PSUbook}) for finite regularity $k$, the odd extension of $\tau_1$, $\Bar{\tau_1} \in C^{k-4}(\partial SM_1)$ , $\Bar{\tau_1}|_{\partial_+ SM_1 } = \tau_1|_{\partial_+SM_1} \in C^{k-4}(\partial_+SM_1)$. This shows that $w_1 \in C^{\min{(m,k-4)}}(\partial_+SM_1)$.

Since $SM$ is away from $\partial_0 SM_1$, we have that $\tau_1|_{SM} \in C^{k-1}(SM)$. Consider $w_1^\# = w_1(\varphi_{(1, \tau_1(x,v)})|_{SM} \in C^{\min{(m,k-4})}(SM)$.
Define $w := w_1^\#|_{\partial_+SM}$, clearly $w^\# = w_1^\#|_{SM}$ since they both agree on $\partial_+SM$ and are constant along geodesics, this shows that $w^\# \in C^{\min{m,k-4}}(SM)$, and so $w \in C_\alpha^m(\partial_+ SM)$.

Now it remains to prove that $I^* w = f$, we have that for all $x \in M$ $$I^*w(x) = \int_{S_xM} w^\#(x,v) dS_x(v) = \int_{S_xM} w_1^\#(x,v) dS_x(v)$$ $$=(I_1^* w_1)(x) = I_1^*I_1 \phi_1 h (x) = Ah(x) = f (x)$$.

\end{proof} \hfill \\

\section{Boundary determination from scattering relation}
\hfill \\

We are now ready to prove theorem 3. We first prove the finite regularity version of (Theorem 5.1.1 of \cite{PSUbook}] and (Theorem 1.6 in \cite{PesUhl2005}).

\subsection{Geometric Preliminaries (cont.)}

\begin{definition}
    Let $\nu$ be the inward pointing normal vector. Define $\partial_{\pm}SM := \{(x,v) \in \partial SM  | \pm \langle v, \nu \rangle \geq 0 \}$, also define the \textit{glancing region} $\partial_0 SM = \partial_+ SM \bigcap \partial_-SM$.
\end{definition}

\begin{remark}
    Note that if we have two metrics $g_1$ and $g_2$ with the same boundary distance function, then by virtue of \ref{thm1} the sets above are all the same.
\end{remark} 

\begin{definition}[Exit time function]
    Let $(M,g)$ be a simple surface. Define $\tau(x,v): SM \rightarrow \mathbb{R}$ the \textit{exit time function}, defined by the length of the (unique) geodesic starting at $x$ in the direction of $v \in S_xM$ and ends at the boundary.

    The non-trapping condition of a simple manifold says precisely that the exit time function is bounded. And the strict convexity condition implies $\tau(x,v) = 0$ for $(x,v) \in \partial_0 SM$.
    
    An argument similar to (lemma 3.2.3 in \cite{PSUbook}) shows that $\tau$ is $C^{k-1}$ away from $\partial_0 SM$.
\end{definition}

\begin{definition}[Scattering relation]
    Define the \textit{odd extension of the exit time function} $\Tilde{\tau}(x,v) = \tau(x,v) - \tau(x,-v)$

    Define the \textit{scattering relation} $\alpha(x,v) : \partial SM \rightarrow \partial M$ to be $$\alpha(x,v) := ( \varphi_{\Tilde{\tau}(x,v)}(x,v) , \varphi'_{\Tilde{\tau}(x,v)}(x,v))$$ where $(\varphi, \varphi')$ is the geodesic flow on $SM$. Clearly $\alpha : \partial_+ SM \rightarrow \partial_- SM$ and vice versa, and $\alpha^2 = id$.

\end{definition}

\begin{definition}
    Let $w \in C(\partial_+SM)$, define $w^\# \in C(SM)$ by $w(\varphi_{\tau{(x,v)}}(x,v))$.

    Also define the odd and even continuation of $w$:

 $$A_{\pm}w(x,v)= 
\begin{cases}    
   w(x,v) & (x,v) \in \partial_+SM \\
   \pm w \circ \alpha (x,v) & (x,v) \in \partial_- SM
\end{cases}$$

Equip $\partial_+ SM$ with the $L^2$ inner product $\int_{\partial_+SM} u v \mu d\Sigma$, with $\mu = \langle \xi, \nu \rangle$ and $d\Sigma = d(\partial M) \wedge  d(S_xM) $ (\cite{PesUhl2005})

Also equip $\partial SM$ with a similar $L^2$ structure with $\int_{\partial_+SM} u v |\mu| d\Sigma$. Then it can be shown (Lemma 9.4.5 \cite{PSUbook}) that $A_{\pm} : L^2_\mu (\partial_+SM) \rightarrow L^2_{\mu}(\partial SM)$ is a bounded operator, and the adjoint $A^*$ is given by $A^*_{\pm} u = ( u \pm u \circ \alpha)|_{\partial_+ SM}$.

\end{definition}

\begin{definition}
    
Define the spaces $$C_\beta^j(\partial_+SM) := \{ w \in C^j(\partial_+SM) : A_+ w \in C^j(\partial SM) \}$$

$$C_\alpha^j(\partial_+SM) := \{ w \in C^j(\partial_+SM) : w^\# \in C^j(SM) \}$$

\end{definition}

\begin{definition}
    We define the \textit{Hilbert transform} $$Hu(x,\xi) = \frac{1}{2\pi} \int_{S_xM}  \frac{1+(\xi,\eta)}{(\xi_\perp , \eta)}, \ \ \ \ \ \xi \in S_xM$$

Also denote the odd and even part of the Hilbert transform $H_+$ and $H_-$ respectively, note that $H_+ u = Hu_+$ and $H_- u = Hu_-$.
\end{definition}

\begin{definition}
    For a $C^k$ metric $g$ We also define the Geodesic vector field $X: C^m(SM) \rightarrow C^{\min{m,k-1}}(SM)$ given by $$Xu(x,\xi) = \frac{d}{dt} (u (\varphi_t(x,v)))|_{t=0}$$

    where $\varphi$ is the geodesic flow.

    Also define $X_\perp: C^m(SM) \rightarrow C^{\min{m,k-1}}(SM)$ given by $$X_\perp u(x,\xi) = \frac{d}{dt} (u (\psi_t(x,v)))|_{t=0}$$
    where $\psi_t (x,v) = (\gamma_{x,v_\perp} (t), W(t))$, where $v_\perp$ is the $90^\circ$ clockwise rotation (This is well defined since our manifold is orientable and 2D), and $W(t)$ is the parallel transport of $v$ along the geodesic $\gamma_{x,v_\perp}$ 
\end{definition}

Finally we define the DN map:

\begin{definition}
    Since we are working with simple surfaces which are diffeomorphic to closed disk $\mathbb{D} \subset \mathbb{R}^2$, we may assume a global coordinate on $M$. Let $0<\lambda<1$ and $f \in C^{2,\lambda}(\partial M)$ and assume the metric $g$ is at least $C^3$, then by theorem 6.14 \cite{gbtrud} there is a unique harmonic $u \in C^{2,\alpha}(M)$ with 
    $$ \Delta u = 0 , \ \ \ \ \ \ u|_{\partial M} =f $$
    
    Define the Dirichlet to Neumann (DN) map $\Lambda: C^{2,\lambda}(\partial M) \rightarrow C^{1,\lambda}(\partial M)$ by $\Lambda f = \partial_{\nu}u$

\end{definition}

\begin{remark}
    Note that by virtue of \ref{thm1}, if two metrics have the same boundary distance function it also implies (after possibly applying a boundary fixing diffeomorphism with one regularity higher than that of the metric) they have the same inward pointing normal vector. 
\end{remark}

\subsection{From surjectivity of $I^*$ to scattering relation}
\hfill \\

Since $A_+w = w^\#|_{\partial_SM}$, it is clear that $C_\alpha^j(\partial_+SM) \subset C_\beta^j(\partial_+SM)$, the theorem below shows a partial converse:

\begin{theorem}\label{fold}
    Let $g \in C^k$, $k>5$,  then $C_\beta^{2j}(\partial_+SM) \subset C_\alpha^{\left \lfloor{\frac{\min{j,k-5}}{2}} \right \rfloor}(\partial_+SM)$
\end{theorem}

We need a couple of technical lemmas for the proof of theorem 6.

\begin{lemma}[lemma 3.2.9 in \cite{PSUbook}]\label{lemma329}
    Let $g$ be a $C^k$ metric on a compact smooth manifold with boundary $M$, let $(x_0, v_0) \in \partial_0 SM$, and let $\partial M$ be strictly convex near $x_0$. Assume that $M$ is embedded in a compact manifold $N$ without boundary. Then, near $(x_0, v_0)$ in $SM$, one has 
\begin{equation}
    \tau(x,v) = Q(\sqrt{a(x,v)}, x,v), 
    -\tau(x,-v) = Q(-\sqrt{a(x,v)}, x,v)
\end{equation}

Where $Q$ is $C^{k-5}$ near $(0,x_0,v_0) \in \mathbb{R} \times SN$ and $a$ is $C^{k-2}$ near $(x_0,v_0) \in SN$.

\end{lemma}
\begin{proof}
    The lemma follows from a simple regularity counting argument in the proof of the smooth metric setting.
\end{proof}

\begin{lemma}[Whitney(\cite{whitney}]\label{lemmawhit}
    Suppose $f \in C^{2k}(\mathbb{R})$ and $f(t) = f(-t)$ for all $t \in \mathbb{R}$, then there exists $h \in C^k$ with $f(t) = h(t^2)$ for all $t \in \mathbb{R}$.
\end{lemma}

\begin{proof}
    This follows from Whitney's proof in \cite{whitney} that if $f \in C^{2k}$ and even then $f(\sqrt{x}) \in C^k$, and the fact that for every sequence $r_i$, there exists a smooth function whose $ith$ derivative at $0$ is $r_i$ (Exercise 8C.2 \cite{folland}.

\end{proof}

\begin{proof}[Proof of theorem \ref{fold}]
    This will be a modified version of the proof in the smooth setting for Theorem 5.1.1 in \cite{PSUbook}. We embed $(M,g)$ isometrically into a closed manifold $(N,g)$ with the same dimension with metric of the same regularity. Let $A_+ w \in C^j(\partial_+ SM)$, extend $A_+w$ to some $W \in C^j$. Consider $F(t,x,v) = \frac{1}{2}W(\varphi_t(x,v))$, then $$w^\#(x,v) = \frac{1}{2}[W(\varphi_{\tau(x,v)}(x,v) + W(\varphi_{-\tau(x,-v)}(x,v)$$  $$ = F(\tau(x,v),x,v) + F(-\tau(x,v),x,v)$$. A similar proof for that of (Lemma 3.2.3 \cite{PSUbook}) show that $\tau$ is $C^{k-2}$ (for $k>2$) away from the glancing region $\partial_0 SM$, so the regularity of $w^\#$ is determined by that near the glancing region.

    Fix some $(x_0, v_0) \in \partial_0 SM$, by \ref{lemma329} above, for $(x,v)$ near $(x_0,v_0)$ we can write $w^\# (x,v) = F(Q(\sqrt{a(x,v),x,v))} + F(Q(-\sqrt{a(x,v)},x,v)$ with $Q$ being $C^{k-5}$ near $(0,x_0,v_0) \in \mathbb{R} \times SN$ and $a$ is $C^{k-2}$ near $(x_0,v_0) \in SN$.
    
    Set $G := F(Q(r,x,v),x,v)$ so we have that near $(x_0,v_0)$ we have $w^\# (x,v) = G(\sqrt{a(x,v)},x,v) + G(-\sqrt{a(x,v)},x,v)$. Clearly $G(r,x,v) + G(-r,x,v)$ is $C^{\min{k-5,m}}(\mathbb{R})(\mathbb{R} \times SN)$ near $(0,x_0,v_0)$ and even in $r$, so we may apply \ref{lemmawhit} above to obtain $H \in C^{\left \lfloor{\frac{\min{j,k-5}}{2}} \right \rfloor}(\mathbb{R} \times SN)$ near $(0,x,v,)$ such that $G(r,x,v) + G(-r,x,v) = H(r^2,x,v)$, which implies near $(x_0,v_0)$ we have $$w^\#(x,v)= G(\sqrt{a(x,v)},x,v) + G(-\sqrt{a(x,v)},x,v) = H(a(x,v),x,v)$$

    which shows that $w^\#$ is $C^{\left \lfloor{\frac{\min{j,k-5}}{2}} \right \rfloor}(SM)$ near $(x_0,v_0)$ in $SM$. The regularity away from the glancing region is $k-1$, so $w^\# \in C^{\left \lfloor{\frac{\min{j,k-5}}{2}} \right \rfloor}(SM)$.

\end{proof}

We now state one final technical theorem we need to prove theorem 3. Following the set up for (Theorem 1.6 \cite{PesUhl2005}).
Let $w \in C_\alpha^{(2,\lambda)}(\partial_+SM)$, if we assume $g$ is a simple $C^3$ metric the argument in Pestov Uhlmann remains valid to show that for $f \in C^{(2,\lambda)}(M)$

\begin{equation}\label{pu1.2}
    I X f = -A_-^* f^0
\end{equation}
where $f^0 = f|_{\partial M}$, from an application of the Hilbert transform (Theorem 1.5 in \cite{PesUhl2005}) we also have:

\begin{equation}\label{pu1.5}
2\pi A^*_{-} H_+ A_+ w = I X_\perp I^*w
\end{equation}

If $h = I^*w \in C^{(2,\lambda)}(M)$, and $h_*$ its harmonic conjugate, then $IX_\perp h = IX h_*$, so \ref{pu1.2} and \ref{pu1.5} together gives

\begin{equation}\label{pu1.6}
2\pi A^*_{-} H_+ A_+ w = -A^*_- h_*^0
\end{equation}

We now prove a converse of this result, following Theorem 1.6 in \cite{PesUhl2005}.

\begin{theorem}\label{thm1.6}
    Suppose $w \in C_\alpha^{(2,\lambda)}(\partial_+ SM)$.  $h_* \in C^{(2,\lambda)}(M)$ the harmonic continuation of $h_*^0 \in C^{(2,\lambda)}(\partial M)$. Then $h := I^* w$ and $h_*$ are harmonic conjugates if and only if \ref{pu1.6} holds.
\end{theorem}

\begin{proof}
    By \ref{pu1.2} and \ref{pu1.5}, \ref{pu1.6} above is equivalent to $IX_\perp h = IXq$ where $q$ is any $C^{(2,\lambda)}$ continuation of $h_*^0$. So $I(\nabla q + \nabla_\perp h) = 0$, since $g$ is $C^3$ simple, it is in particular $C^{1,1}$ simple (\cite{IlmKyk2023}), so by the injectivity of lipscthiz 1 form  for $C^{1,1}$ simple manifolds, the vector field $\nabla q + \nabla_\perp h = \nabla p$ for $p \in C^{1,1}(M)$ and $p|_{\partial M}= 0$.
    (Note : The injectivity of Lipschitz 1-form with arbitrary boundary conditions on $C^{1,1}$ simple manifolds follows from theorem 1 (b) in \cite{IlmKyk2021} and lemma 2 in \cite{IlmKyk2023}, see proof of theorem 1 in \cite{IlmKyk2023}). 

    Since $q$ and $h$ are $C^{2,\lambda)}$, their gradients are $C^{1,\lambda}$, which implies $p \in C^{(2,\lambda)}(M)$ as well. Now consider the function $h_* := q-p$, then $h_*$ is in $C^{2,\lambda)}$ and is the harmonic continuation of $h_*^0$ since $h_*^0|_{\partial M} = q|_{\partial M} = h_*^0$, and $h$ and $h_*$ are harmonic conjugates by constructions. 
\end{proof}

We are now ready to prove theorem 3.
 
\begin{proof}[Proof of theorem 3]
\hfill \\

    Let $g_1$ and $g_2$ be two $C^{16}$ simple metrics on a compact two dimensional manifold with boundary $M$, so that $d_{g_1} = d_{g_2}$. By theorem 1 there exists a boundary fixing gauge $\Phi: M \rightarrow M$ such that $g_1|_{\partial M} = \Phi ^* g_2 |_{\partial M}$. Since DN map is invariant in two-dimensional under such a gauge (See the beginning of 11.6 in \cite{PSUbook}), we will denote $\Phi^*g_2$ simply by $g_2$ from here on (Since proving DN maps for $g_1$ and $\Phi^*g_2$ are equal implies equality for DN map of $g_2$). From here on out we will use subscript $1,2$ to denote all geometric objects and operators that depend on the metrics.
    
    Suppose $l=9$, $m=6$ so that $k-7= 9 > l-1 =8 > +1 = 7$. Given $h_*^0 \in C^{10}(\partial M)$, let $h_{*,1} \in C^{10}(M)$ its harmonic continuation with respect to $g_1$ and $h_1 \in C^{10}(M)$ its harmonic conjugate. By theorem 2 \ref{thm2} we can find $w \in C_\alpha^6(\partial_+SM))$ such that $I_1^*w= h_1$. By the analysis above we have that \ref{pu1.6} holds for $g_1$.

    Note that $A,A^*_{-},A_+$ are all determined by the scattering relation, so by assumption they are the same for both metrics. $H_+$ applied to the function $ (A_+w) \in C(\partial SM)$ is an integral over $S_xM$ which is the same for both metrics since $g_1|_{\partial M} =g_2|_{\partial M}$ by \ref{thm1}.

So \ref{pu1.6} holds for $g_2$ as well.

Clearly $w \in C_\alpha^6(\partial_+SM) \subset  C_\beta^6(\partial_+SM)$ for $g_1$, but $C_\beta^6(\partial_+SM)$ is the same for both metrics since it is determined by the regularity of $A_+ w$ on $\partial SM$, so we have that $w \in C_\beta^6(\partial_+SM)$ for $g_2$ as well, now apply \ref{fold} and conclude that for $g_2$ we have:

$$w \in C_\beta^6(\partial_+SM) \subset C_\alpha^3(\partial_+SM)$$

Since \ref{pu1.6} holds, we can apply theorem \ref{thm1.6} to conclude that the function $I_2 ^* w \in C^3(M)$, and any $C^3$ harmonic continuation $h_{*,2}$ with respect to $g_2$ of $h_*^0$ are $C^{(2,\lambda)}$ harmonic conjugates.

We know that $h_i^0 := h_i|_{\partial M}$. Now the $g_1$ DN map applied to $h_*^0$ is $$\Lambda_1 h_*^0= \langle \nu, \nabla h_{*,1}|_{\partial M} \rangle = \langle \nu_\perp , \nabla h_{1}^0|_{\partial M} \rangle = \partial_{\nu_\perp}h_1^0$$.

But note that $h_1^0 = h_2^0$ since $I_i^*w |_{\partial M} = \int_{S_xM} A_+ w (x,\xi) dS_xM$ which is the same for both metrics, this shows that as $C^{1,\lambda}(\partial M)$ functions, $\partial_\nu h_{1,*} = \partial_\nu h_{2,*}$ are equal, hence equal as functions of their maximum regularity. This concludes theorem 3.

\end{proof}

\begin{remark}
    We conclude this paper by noting that the only step left for proving boundary rigidity is the Calder\'{o}n problem for $C^{17}$ metrics in two dimension. The Calder\'{o}n problem in 2D was resolved by Lassas and Uhlmann \cite{LaU} in 2D and for real analytic metrics in higher dimension, this result was later generalized to complete manifolds in \cite{LTU}. One promising approach for proving boundary rigidity in the 2D case, is to note that in \cite{LaU} they make use of a result in \cite{LeU} to show that the DN map determines the metric at the boundary in the tangential direction- which is not needed if we are only interested in boundary rigidity by \ref{thm1}. Another possible approach can be used to prove the full Calder\'{o}n problem in 2D for metrics with regularity as low as $C^{1,\alpha}$, is to modify the later proof by Belishev \cite{Belv}, which characterized the complex structure on a manifold by the algebra of holomorphic functions, which in turns determines the conformal classes of metric on $M$. Both proofs can be adapted to the $C^{1,\alpha}$ case since they both rely on the existence of isothermal coordinates \cite{chern} and the analysis of the induced complex structures. It is currently in work at the time of writing of this paper.
\end{remark}

\bibliographystyle{plain}
\bibliography{references}

\begin{thebibliography}{10}

\bibitem{anderson}
Lars Andersson and Piotr~T. Chru\'{s}ciel.
\newblock Solutions of the constraint equations in general relativity satisfying "hyperboloidal boundary conditions".
\newblock {\em Polska Akademia Nauk}, 1996.

\bibitem{Belv}
M.I. Belishev.
\newblock The calder\'{o}n problem for two-dimensional manifolds with the bc-methods.
\newblock {\em SIAM J. Math. Anal.}, 35(1):172--182, 2003.

\bibitem{folland}
Gerald~B. Folland.
\newblock {\em Introduction to Partial Differential Equations. Second edition}.
\newblock Princeton University Press, Princeton, N.J., 1995.

\bibitem{gbtrud}
David Gilbarg and Neil~S. Trudinger.
\newblock {\em Elliptic Partial Differential Equations of Second Order, Second edition}.
\newblock Springer Berlin, Heidelberg, 2001.

\bibitem{IlmKyk2021}
Joonas Ilmavirta and Antti Kykk\"anen.
\newblock {Pestov identities and X-ray tomography on manifolds of low regularity}.
\newblock {\em Inverse Problems and Imaging}, December 2021.
\newblock To appear.

\bibitem{IlmKyk2023}
Joonas Ilmavirta and Antti Kykk\"anen.
\newblock {Tensor tomography on negatively curved manifolds of low regularity}.
\newblock March 2023.

\bibitem{IKL2023}
Joonas Ilmavirta, Antti Kykk\"anen, and Kelvin Lam.
\newblock {Microlocal analysis of the x-ray transform in non-smooth geometry}.
\newblock https://arxiv.org/abs/2309.12702 Sep 2023.

\bibitem{LTU}
Matti Lassas, Michael Taylor, and Gunther Uhlmann.
\newblock The dirichlet-to-neumann map of complete manifolds with boundary.
\newblock {\em Comm. Anal. Geom.}, 11(2):207--221, 2003.

\bibitem{LaU}
Matti Lassas and Gunther Uhlmann.
\newblock On determining a riemannian manifold from the dirichlet-to-neumann map.
\newblock {\em Ann. Sci. Ecole Norm. Sup.. Series 4}, 34(5):771--787, 2001.

\bibitem{irm}
John~M. Lee.
\newblock {\em Introduction to Riemannian Manifolds}.
\newblock Graduate Texts in Mathematics. Springer Cham, 2018.

\bibitem{LeU}
John~M. Lee and Gunther Uhlmann.
\newblock Determining real-analytics anisotropic conductivites by boundary measurements.
\newblock {\em Comm. Pure Appl. Math.}, 42(8):1097--1112, 1989.

\bibitem{marschall1996}
J.~Marschall.
\newblock Nonregular pseudo-differential operators.
\newblock {\em Z. Anal. Anwendungen}, 15(1):109--148, 1996.

\bibitem{Minguzzi}
E.~Minguzzi.
\newblock Convex neighborhoods for lipschitz connections and sprays.
\newblock {\em Springer Vienna}, 177(1):569--625, 2015.

\bibitem{PSUbook}
Gabriel~P. Paternain, Mikko Salo, and Gunther Uhlmann.
\newblock {\em Geometric inverse problems---with emphasis on two dimensions}, volume 204 of {\em Cambridge Studies in Advanced Mathematics}.
\newblock Cambridge University Press, Cambridge, 2023.
\newblock With a foreword by Andr\'{a}s Vasy.

\bibitem{PesUhl2005}
Leonid Pestov and Gunther Uhlmann.
\newblock Two dimensional compact simple {R}iemannian manifolds are boundary distance rigid.
\newblock {\em Ann. of Math. (2)}, 161(2):1093--1110, 2005.

\bibitem{chern}
Chern Shiing-Shen.
\newblock An elementary proof of the existence of isothermal parameters on a surface.
\newblock {\em Proc. Amer. Math. Soc.}, 6(5):771--782, Oct 1955.

\bibitem{whitney}
Hassler Whitney.
\newblock Differentiable even functions.
\newblock {\em Duke Math. J.}, 10(1):159--160, March 1943.

\end{thebibliography}

\end{document}